%% file: symplectic.tex
\documentclass[12pt]{amsart}
\input macros.ltx

\usepackage{hyperref}

\begin{document}

\title{Symplectic generic complex structures on $4$-manifolds with $b_+=1$}
\date{\today}

\author{Paolo Cascini}
\address{Department of Mathematics\\
Imperial College London\\
London SW7 2AZ, UK}
\email{p.cascini@imperial.ac.uk}
\author{Dmitri Panov}
\address{Department of Mathematics\\
King's College London\\
Strand,
London
WC2R 2LS, UK}
\email{dmitri.panov@kcl.ac.uk}

\begin{abstract}
We study symplectic structures on K\"ahler surfaces with $p_g=0$.
We give an example of a projective surface which admits a symplectic structure which is not compatible with any K\"ahler metric.
\end{abstract}

\maketitle

\section{Introduction}
The main purpose of this note is to give a negative answer to a question raised by Tian-Jun Li \cite{Li08}:

\begin{question}\label{question}Let $X$ be a closed, smooth, oriented $4$-manifold
which underlies a K\"ahler surface such that $p_g(X)=0$.  Does $X$ admit a symplectic generic complex structure?
\end{question}

A complex structure $J$ on $X$ is called {\em symplectic generic} if for any symplectic form $\omega$ of $X$ such that $-c_1(X,\omega)$ coincides with the canonical class $K_J$ of $J$, there exists a K\"ahler form $\omega'$ cohomologous to $\omega$.

One of the main motivations for this question is the fact that, by a result of Biran \cite{Biran99}, the existence of a symplectic generic complex structure on any rational $4$-manifold implies the famous Nagata's conjecture (see \cite{Li08} for more details).
Recall that a smooth $4$-manifold $X$ is said to be  {\em rational} if it is diffeomorphic to either $S^2\times S^2$ or $\mathbb C\mathbb P^2 \# k \overline{\mathbb C \mathbb P^2}$, for some $k\ge 0$.

On the other hand, if $X$ is the $4$-manifold underlying a smooth minimal projective surface of general type (i.e. with big and nef canonical line bundle) then there exists a symplectic form inside the class of the canonical line bundle of $X$ (see \cite{Cat09,STY02}). Therefore, if $p_g(X)=0$, the existence of a symplectic generic complex structure on $X$ would, in particular, imply the existence of a K\"ahler-Einstein metric with negative curvature on $X$,  by the result of Aubin and Yau.
For example,
Catanese and LeBrun \cite{CL97} showed the existence
of a K\"ahler-Einstein metric with negative curvature on the generic Barlow surface, which is a projective surface of general type homeomorphic to  $\mathbb C\mathbb P^2 \# 8 \overline{\mathbb C \mathbb P^2}$. But the question remains a hard problem in general, as a classification of the projective surfaces with zero genus is still beyond our reach (see the recent survey \cite{BCP10} for an updated account).

Our example is obtained by considering the $4$-manifold $X=(\Sigma\times S^2)\# \overline{\mathbb C \mathbb P^2}$, where $\Sigma$ is a Riemannian surface of genus one. We show the existence of a symplectic form on $X$ which is not cohomologous to any K\"ahler form on $X$, with respect to any complex structure $J$. From an algebraic geometric point of view, this corresponds to saying that the Seshardi constant of a suitable ample class on any uniruled projective surface over an elliptic curve is not maximal (e.g. see \cite{G06}). In particular, it follows that $X$ does not admit a symplectic generic complex structure.

Moreover, we describe a  minimal surface of general type, for which the underlying manifold does not admit a symplectic generic complex structure. The construction relies on a recent result by Bauer and Catanese \cite{BC09}.

Note that both these examples have infinite fundamental group.

{\bf Acknowledgements}. We would like to thank Valery Alexeev and Burt Totaro four useful
conversations. The first author is partially supported by an EPSRC grant. The second author is
supported by a Royal Society University Research Fellowship.

\section{Preliminary results}

In this section, we recall some basic definition and well known facts about the space of symplectic forms on a smooth $4$-manifold.

Given a closed smooth oriented $4$-manifold $X$, we consider the  {\em positive cone} of $X$, which is defined as the set
$$\mathcal P_X=\{a \in H^2(X,\mathbb R)\mid a^2>0\}.$$

Moreover, we denote by $\Omega_X$ the space of orientation-compatible symplectic forms on $X$. Let
$$\mathcal C_X=\{[\omega]\mid \omega \in \Omega_X\}\subseteq H^2(X,\mathbb R)$$ and let $K_\omega=-c_1(X,\omega)$ be the canonical  class of $\omega\in\Omega_X$. We denote by $\mathcal K_X$ the union of all elements $K_\omega$ in $H^2(X,\mathbb Z)$, where $\omega\in \Omega_X$.
For any $K\in \mathcal K_X$,  let
$$\mathcal C_{(X,K)}=\{[\omega]\in \mathcal C_X\mid K_\omega=K  \}.$$
If $K$ is a torsion class, then we replace $C_{(X,K)}$ by its intersection with the component of $\mathcal P$ which contains any K\"ahler class.
Note that a complex structure $J$ on $X$ is symplectic generic if $\mathcal C_J=\mathcal C_{(X,K_J)}$, where $\mathcal C_J$ denotes the K\"ahler cone of $J$ and $K_J$ is the canonical class of $J$.

Let $\mathcal E_X$ be the set of cohomology classes whose Poincar\'e
dual are represented by smoothly embedded spheres of self-intersection $-1$.
In particular, $X$ is said to be  {\em minimal} if $\mathcal E_X$ is empty.
Moreover, for any  $K\in H^2(X,\mathbb Z)$, we denote
$$\mathcal E_{(X,K)}=\{E\in \mathcal E_X\mid E\cdot K=-1 \}.$$

The following result by Li and Luo \cite{LiLuo01} will play an important role:
\begin{theorem}\label{l_ll}
Let $(X,\omega)$ be a closed, symplectic $4$-manifold with $b_+(X)=1$.

Then
$$\mathcal C_X=\{a\in \mathcal P_X \mid a\cdot E \neq 0 \quad \text{for all } E\in \mathcal E_X  \}.$$
Let $K\in \mathcal K_X$. Then $\mathcal C_{(X,K)}$ is contained in one of the
components of $\mathcal P_X$, denoted by $\mathcal P_{(X,K)}$. Moreover,
$$\mathcal C_{(X,K)}=\{a\in \mathcal P_{(X,K)}\mid a \cdot E>0 \quad \text{for all } E\in \mathcal E_{(X,K)}\}.$$
\end{theorem}
\begin{proof}
See \cite[Theorem 4]{LiLuo01} and \cite[Theorem 3.11]{Li08}.
\end{proof}

\begin{lemma}\label{l_knm}
Let $(X,J)$ be a minimal complex surface with $b_+(X)=1$ and  which admits a K\"ahler class $[\omega]\in \mathcal C_J$. Then $J$ is a symplectic generic complex structure if and only if any $J$-holomorphic curve in  $X$ has non-negative self-intersection.
\end{lemma}

\begin{proof}
By the K\"ahler Nakai-Moishezon criterion \cite{Buch99,L99}, if the K\"ahler cone $\mathcal C_J$ is not empty then it coincides with the set of elements in $\mathcal P_{(X,K_J)}$ which are positive on every $J$-holomorphic curve with negative self-intersection.
Thus, if there is no such a curve on $X$, it follows that $J$ is a symplectic generic complex structure.

Let us assume now that $C$ is a $J$-holomorphic curve with negative self-intersection. Let $v=\omega(C)$ and $m=-C^2$ and define $a(t)=[\omega] + t PD(C)\in H^2(X,\mathbb R)$ for any  $t\ge 0$. Then, since\
$$a(t)^2=[\omega]^2+2t\omega(C)+t^2C^2> 2tv - t^2m,$$
it follows that there exists $T>v/m$ such that $a(T)\in \mathcal P_{(X,K_J)}$.  Since $X$ is minimal, Theorem \ref{l_ll} implies that $a(T)$ is represented by a symplectic form $\omega_T$ such that $K_{\omega_T}=K_J$. On the other hand, $\omega_{T}(C)=v-Tm<0$, thus $a(T)$ is not a K\"ahler class. In particular, $J$ is not a symplectic generic complex structure.
\end{proof}

By the K\"ahler Nakai-Moishezon criterion and Theorem \ref{l_ll}, it also follows that a positive answer to Question \ref{question} in the case
of rational $4$-manifolds is equivalent to the following  conjecture (Harbourne-Hirschowitz):
any integral curve with negative self-intersection on
the blow-up of $\mathbb CP^2$ at a set of points in very general position
is a smooth rational curve with self-intersection $-1$.

\section{Ruled Manifolds}

In this section, we show the existence of a smooth uniruled complex manifold, which does admit a symplectic generic complex structure.
\begin{lemma}\label{l_curve}
Let $\Sigma$ be an elliptic curve, and let $p\colon \map Y.\Sigma.$ be a minimal ruled surface over $\Sigma$, such that the parity of the intersection pairing on $H^2(Y,\mathbb Z)$ is odd.   Let $X$  be the blow-up of $Y$ at one point $\eta\in Y$. Let $k$ be the canonical class of $X$,
and let $e$ be the class of the exceptional divisor.

Then the class
$e-2k$
contains an effective curve.
\end{lemma}
\begin{proof}

By Atiyah's classification \cite{Atiyah57} of rank 2 vector bundles on an elliptic
curve, it follows that $Y=\mathbb P(\mathcal E)$ where $\mathcal E$ is either
the indecomposable vector bundle  contained in the sequence
$$0\to \ring \Sigma.\to \mathcal E \to \ring \Sigma.(p)\to 0$$
for some $p\in \Sigma$  or $\mathcal E=\ring \Sigma.\oplus \ring \Sigma.(L)$ where $L$ is a line bundle of odd degree $m<0$.

Let us consider first the case of the indecomposable vector bundle. It is known (e.g. see \cite{CC93}) that in this case $\mathbb P(\mathcal E)$ is isomorphic to the symmetric product $S^2\Sigma$ of the elliptic curve $\Sigma$, i.e. the quotient of $\Sigma\times \Sigma$ by the natural action of $\mathbb Z/2\mathbb Z$. We will denote by $[x,y]\in S^2\Sigma$ the class of an element $(x,y)\in \Sigma\times \Sigma$.
Note that the projection $p\colon\map S^2\Sigma.\Sigma.$ is defined by
 $p([x,y])=x+y$.
Consider the family of curves
$$C_t=\{[x,t+x]\mid x\in \Sigma\} \qquad\text{for any }t\in \Sigma.$$
If $t\in \Sigma$ is not a $2$-torsion point, then the curve $C_t$ is a smooth elliptic curve. Otherwise, $C_t$ is a non-reduced elliptic curve.
Note that, for any $s,t\in \Sigma$, we have $C_t=C_s$ if and only if $t=s$ or $t=-s$ and
$C_t$ and $C_s$ are disjoint otherwise. It follows that $C_t^2=0$. Moreover, given $s,t\in \Sigma$, there exist exactly $4$ points $x\in \Sigma$ such that $2x+t=s$. Thus, if $t$ is a general point in $\Sigma$, then the general fiber of $p$ meets $C_t$ in exactly $4$ points. Let $f$ be the numerical class of the pull-back of the general fiber of $p$ in $X$ and let $\delta$ be the numerical class of the pull-back of $C_t$. Then
$$\delta^2=C^2_t=0\qquad \delta\cdot e=0\qquad\text{and}\qquad \delta\cdot f=4.$$
By adjunction, we have that $k\cdot \delta = -\delta^2=0$. Similarly, we have $k\cdot e=-1$ and $k\cdot f=-2$. Moreover, since $e,f$ and $k$ are a basis of $H^2(X,\mathbb Q)$, it follows easily that
$\delta=2e-2k$. For any point $\eta\in S^2\Sigma$ there exists $t\in \Sigma$ such that
$\eta\in C_t$. If $X$ is the blow-up of $Y$ at $\eta$ and $C'_t$ is the proper transform of $C_t$ in $X$, then $C'_t$ is in the class of $(2-q)e-2k$, where $q\ge 1$ is the multiplicity of $C_t$ at $\eta$. In particular, the class $e-2k$ contains an effective curve, as claimed.

Let us consider now the case of a decomposable vector bundle $\mathcal E=\mathcal O_\Sigma\oplus \mathcal O_\Sigma(L)$ where $L$ is a line bundle on $\Sigma$ of odd degree $m<0$. Then, there exists an holomorphic section $C$ in $Y$ such that $C^2=m$. If $\xi$ is the numerical class of the pull-back of $C$ in $X$, it follows easily that $2\xi= e+mf-k$, where $f$ is the pull-back of the general fiber of $p$. In particular, $e-2k=4\xi + (-2mf -e)$ is the class of a (possibly not irreducible) effective curve in $X$.
\end{proof}

\begin{remark}Note that the uniruled surface which is the projectivization of the decomposable vector bundle can be obtained as a deformation of the projectivization of the indecomposable one. Thus, in the proof of the previous lemma, the second case would follow immediately from the first one.
\end{remark}

\begin{lemma}\label{kodimension} A complex surface $X$ homeomorphic to $(\Sigma\times S^2)\# \overline{\mathbb C \mathbb P^2}$,
is bi-holomorphic to a blow up at a single point of a minimal ruled surface $Y$ over
an elliptic curve, such that the intersection pairing on $H^2(Y,\mathbb Z)$ is odd.
\end{lemma}

\begin{proof} Recall that from the Enriques-Kodaira classification
of complex surfaces, it follows that each complex surfaces with odd $b_+$
is K\"ahler, and that any algebraic surface of non-negative Kodaira dimension and
zero holomorphic Euler characteristics is bi-meromorphic to a torus or
a bi-elliptic surface. Since $b_+(X)=1$, it follows that $X$ is K\"ahler and $p_g(X)=0$. Thus $X$ is algebraic.
Since $\pi_1(X)=\mathbb Z^2$ and $\chi(\ring X.)=0$, we conclude
that $X$ has Kodaira dimension $-\infty$.

By the classification of algebraic surfaces, it follows that if $Y$ is the minimal model of $X$, i.e. the surface obtained after blowing-down all the holomorphic $(-1)$ spheres on $X$, then $Y$ is a uniruled surface over a Riemannian surface $\Sigma$. Since $b_1(Y)=b_1(X)=2$, it follows that the genus of $\Sigma$ is one. Moreover, since $b_2(X)=3$,
it follows that $X$ is the blow-up of a ruled surface  over an elliptic curve at a single point $p\in Y$. In particular $X$ has exactly two holomorphic rational curves $E_1$ and $E_2$ with self-intersection $-1$: one is the exceptional divisor of the blow-up map and the other is the strict transform of the rational fiber passing through the blown-up point.
Assume that the intersection form on $H^2(Y,\mathbb Z)$ has even parity. Let $C$ be a curve on $Y$ which pass through $p$ and which meets the fiber of the fibration $\map Y.\Sigma.$ transversally at $p$. Then the strict transform of $C$ in $X$ has odd self-intersection  and it does not intersect $E_2$. Thus, after contracting $E_2$ we obtain a surface $Y'$ such that the intersection form on  $H^2(Y',\mathbb Z)$ has odd parity.
After replacing $Y$ by $Y'$, we may assume that $H^2(Y,\mathbb Z)$ has odd parity.
\end{proof}

\begin{lemma}\label{EXk} Let $\pi\colon \map Y.\Sigma.$ be a ruled projective surface over an elliptic curve $\Sigma$, such that  $H^2(Y,\mathbb Z)$ has odd parity. Let
 $X$ be the blow up of  $Y$ at a single point.
Let $k$ be the class of
the canonical class of $X$ and let $e_1$, $e_2$ be the classes
of the two rational curves of self-intersection $-1$ on $X$.

Then
$\mathcal E_{(X,k)}=\{e_1,e_2\}$.
\end{lemma}
\begin{proof} Let $e$ be a class in
$H_2(X,\mathbb Z)$ which can be represented by a smoothly embedded sphere in $X$
such that $e^2=-1$. Then $e$ belongs to the kernel of $\pi_*: H_2(X,\mathbb Z)\to H_2(\Sigma,\mathbb Z)$.
This kernel is spanned by $e_1$ and $e_2$ and we deduce $e=\pm(ne_1+ (n-1)e_2)$ for some integer
$n$. At the same time $e_1\cdot k=e_2\cdot k=-1$, since $e_1, e_2$
are the classes of exceptional curves on $X$. Thus, if $e\in \mathcal E_{(X,k)}$, then $e\cdot k=-1$ which implies $e=e_1$ or
$e=e_2$.
\end{proof}

\begin{theorem}
Let $\Sigma$ be a Riemann surface of genus $1$, let $\Sigma\times S^2$ be the trivial $S^2$-bundle on $\Sigma$ and let $X=(\Sigma\times S^2)\#\overline {\mathbb C\mathbb P}^2$.

Then, for any complex structure $J$ on $X$, there exists a symplectic form $\omega$ on $X$ such that $\omega$ is not K\"ahler with respect to  $J$.
Moreover, $X$ does not admit any symplectic generic complex structure.
\end{theorem}
\begin{proof}

Let $J$ be a complex structure on $X$, let $k$ be the canonical class of $(X,J)$ and let $e$ be the class of the exceptional divisor $E$ of the contraction $\map X.Y.$, whose existence is guaranteed by
Lemma \ref{kodimension}.
Let $a$ be the first Chern class of an ample line bundle on $X$. By Lemma \ref{l_curve}, it follows that $v=a\cdot (e-2k)>0$. Let
$$a(t)=a+t(e-2k)\in H^2(X,\mathbb R)\qquad \text{for all }t>0.$$
In particular, $a(t)\cdot (e-2k)=v-t$ and $a(v)^2=a^2+v^2>0$.
Thus, there exists $T>v$ such that $a(T)^2>0$.
Moreover, if $E\in \mathcal E_{(X,k)}$, then
$$a\cdot E>0\qquad   \qquad k\cdot E=-1 \qquad \text{and by Lemma \ref{EXk}}  \qquad e\cdot E\ge -1.$$
Thus, $a(t)\cdot E>0$ for all $t>0$. Since $b_+(X)=1$,
Theorem \ref{l_ll} implies that the class $a(T)$ is represented by a symplectic form $\omega$, such that $K_\omega=k$.

On the other hand, by Lemma \ref{l_curve}, the class $e-2k$ is represented by a $J$-holomorphic curve $C$ such that $a(T)\cdot C < 0$, since $T>v$. Thus, the class $a(T)$ does not contain a K\"ahler form. In particular, $J$ is not a symplectic generic complex structure.
\end{proof}

\section{Non-ruled Manifolds}

In this section we study Question \ref{question} in the case of smooth {\it minimal}
$4$-manifolds with non-negative Kodaira dimension.

\begin{question}\label{conj}
Let $X$ be a {\it minimal} $4$-manifold which underlies a K\"ahler surface such that $p_g(X)=0$.
Does $X$ admit a symplectic generic complex structure?
\end{question}

In particular, we show that the question has positive answer in the case of zero Kodaira dimension and we provide an example of a minimal surface of general type which does not admit a symplectic generic complex structure.

By the Sieberg-Witten theory, the Kodaira dimension of a K\"ahler surface is preserved under diffeomorphism \cite{BHPV04}. As noted in \cite{Li08}, any uniruled $4$-manifold, i.e. a manifold which underlies a K\"ahler surfaces of Kodaira dimension $-\infty$, admits a symplectic generic complex structure.

We first consider the case of zero Kodaira dimension:

\begin{proposition}\label{p_kod0}
Let $X$ be a $4$-manifold which underlies a K\"ahler surface such that $p_g(X)=0$ and $\kod(X)=0$.

Then $X$ admits a symplectic generic complex structure.
\end{proposition}
\begin{proof}
 By the classification of algebraic surfaces, it follows that the canonical class of $X$ is numerically trivial. Thus, by the adjunction formula, the only holomorphic curves of negative self-intersection, are smooth rational curves $C$ such that $C^2=-2$. In particular, Lemma  \ref{l_knm} implies that it is sufficient to show that there exists a complex structure on $X$ which does not admit any of these curves.

 By the classification of algebraic surfaces, we just need to consider two cases: Enriques surfaces and bi-elliptic surfaces.
The moduli space of Enriques surfaces is irreducible and by a result of Barth and Peters \cite[Proposition 2.8]{BP83}, the generic Enriques surface does not contain any smooth rational curve of self-intersection $-2$.

If $X$ is a bi-elliptic surface, then $X=\Sigma_1\times\Sigma_2/G$, where $\Sigma_1$ and $\Sigma_2$ are Riemannian surfaces of genus one and $G$ is an abelian group acting by complex multiplication on $\Sigma_1$ and by translation on $\Sigma_2$. Then the natural action of $\Sigma_2$ on $\Sigma_1\times \Sigma_2$ commutes with the action of $G$ and in particular $\Sigma_2$ acts on $X$ non trivially. Thus, $X$ does not admit any negative self-intersection curve.
By Lemma \ref{l_knm}, it follows that any complex structure on $X$ is symplectic generic.
  \end{proof}

If $X$ is a minimal surfaces of general type with $p_g(X)=0$, it is well known that $q=0$ and $1\le K_X^2\le 9$.  Thus, their moduli spaces is a union of finitely many irreducible varieties.  Nevertheless, it is still not clear what the topology  for these surfaces is (see \cite{BCP10} for a recent survey).  As stated in the introduction, if $X$ is the $4$-manifold underlying the surface $X$,  a positive answer to question \ref{conj} would imply the existence of a complex structure on $X$ which admits a  K\"ahler-Einstein metric.
By the results in \cite{B84,LP07,PPS09,PPS09a}, it follows that there exist a surface of general type which is homeomorphic to $\mathbb C\mathbb P^2 \# k \overline{\mathbb C \mathbb P^2}$, for $5\le k\le 8$. It follows by  \cite{CL97,RS09} that, on any of these surfaces, there exists a complex structure which admits a K\"ahler-Einstein metric with negative curvature.

In general, if $X$ is a minimal surface of general type with $p_g(X)=0$, then $\chi(\ring X.)=1$ and by Noether's formula  we have
$$b_2(X)=\chi(X)-2=12\chi(\ring X.)-K_X^2-2=10-K_X^2.$$
Thus, if $K_X^2=9$, then any class in $\mathcal P_X$ is the multiple of an ample class and the answer to Quesiton
\ref{conj} is obvious.

Let us consider now the case of a surface of general type $S$ with $p_g(X)=0$ and  $K_X^2=8$. All the known examples
have infinite fundamental group and their universal cover is the bidisk $\Delta_1\times \Delta_2\subseteq \mathbb C^2$ \cite{BCP10},
so we assume that $S$ is of this type.
Denote by $w_1$ and $w_2$ two semi-positive (1,1)-forms on $\Delta_1\times \Delta_2$ obtained
via pullbacks of Poincar\'e metrics from the projections of the bidisk to its factors. For any $a,b>0$ the form $aw_1+bw_2$ is K\"ahler on the bidisk and is invariant under the action of $\pi_1(X)$. Thus,
 it descends to a K\"ahler form $w_{a,b}$ on $S$.
Since $b_2(X)=2$, it folllows that for $a,b>0$ the forms $w_{a,b}$ span one of the two connected components
of $\mathcal P_X$, and so the complex structure on $X$ is symplectic generic.

On the other hand, the results in \cite{BC09} immediately imply the existence of a minimal surface of general type which does not admit
a symplectic generic complex structure. Burniat showed the existence of a minimal surface $X$ of general type such that $K_X^2=6$, $p_g(X)=0$, and which admits a $(\mathbb Z/2\mathbb Z)^2$-cover of $\mathbb C\mathbb P^2$ blown-up at $3$ points.
We will call such a surface a {\em Burniat surface}.
\begin{theorem}
Let $X$ be a $4$-manifold which underlies a Burniat surface. Then $X$ does not admit a symplectic generic complex structure.
\end{theorem}
\begin{proof}
By \cite[Theorem 0.2]{BC09},  any complex structure $J$ on $X$ is a Burniat surface. In particular, $X$ admits a $J$-holomorphic curve $C$ of negative self-intersection, which maps to a $(-1)$-curve on the blow-up of $\mathbb C\mathbb P^2$ at $3$ points. More specifically, $C$ is an elliptic curve of self-intersection $-1$. Thus, by Lemma \ref{l_knm}, it follows that $J$ is not symplectic generic.
\end{proof}
Note that a Burniat surface has infinite fundamental group. We do not know any complex surface with $p_g=0$, finite fundamental group and which does not admit a symplectic generic complex structure.

Recall finally that there exist a wide class of
minimal elliptic surfaces of Kodaira dimension $1$ and with $p_g=0$. These surfaces
have topological Euler characteristic equal to $12$, the base of the corresponding elliptic
fibration is $\mathbb CP^1$, and the fibration can have any number of multiple fibers greater than $1$.
It would be interesting to show that all such surfaces admit a symplectic
generic complex structure.

\bibliographystyle{amsalpha}
\bibliography{Library}

\end{document}

%% file: macros.ltx
\usepackage{amsfonts,amsmath,amsthm,amssymb,latexsym,amsthm,newlfont,enumerate}

\makeatletter

\def\jobis#1{FF\fi
  \def\predicate{#1}%
  \edef\predicate{\expandafter\strip@prefix\meaning\predicate}%
  \edef\job{\jobname}%
  \ifx\job\predicate
}

\makeatother

\if\jobis{proposal}%
 \def\try{subsection}%
\else
  \def\try{section}%
\fi

\theoremstyle{plain}
\newtheorem{theorem}{Theorem}[\try]

\newtheorem{lemma}[theorem]{Lemma}

\newtheorem{proposition}[theorem]{Proposition}
\newtheorem{definition-lemma}[theorem]{Definition-Lemma}
\newtheorem{question}[theorem]{Question}
\newtheorem{red-question}[theorem]{\textcolor{red}{Question}}

\theoremstyle{definition}

\newtheorem{remark}[theorem]{Remark}

\def\ideal#1.{I_{#1}}
\def\ring#1.{\mathcal {O}_{#1}}

\def\fring#1.{\hat{\mathcal {O}}_{#1}}
\def\proj#1.{\mathbb {P}(#1)}
\def\pr #1.{\mathbb {P}^{#1}}
\def\dpr #1.{\hat{\mathbb {P}}^{#1}}
\def\af #1.{\mathbb A^{#1}}
\def\Hz #1.{\mathbb F_{#1}}
\def\Hbz #1.{\overline{\mathbb F}_{#1}}
\def\fb#1.{\underset #1 {\times}}
\def\rest#1.{\underset {\ \ring #1.} \to \otimes}
\def\au#1.{\operatorname {Aut}\,(#1)}
\def\deg#1.{\operatorname {deg } (#1)}

\def\pic#1.{\operatorname {Pic}\,(#1)}
\def\pico#1.{\operatorname{Pic}^0(#1)}
\def\picg#1.{\operatorname {Pic}^G(#1)}
\def\ner#1.{NS (#1)}
\def\rdown#1.{\llcorner#1\lrcorner}
\def\rfdown#1.{\lfloor{#1}\rfloor}
\def\rup#1.{\ulcorner{#1}\urcorner}
\def\rcup#1.{\lceil{#1}\rceil}

\def\n1#1.{\operatorname {N_1}(#1)}  
\def\cn1#1.{\overline{\operatorname {N^1}(#1)}} 
\def\cone#1.{\operatorname {NE}(#1)}     
\def\ccone#1.{\overline{\operatorname {NE}}(#1)}
\def\none#1.{\operatorname {NF}(#1)}
\def\cnone#1.{\overline{\operatorname {NF}}(#1)}
\def\mone#1.{\operatorname {NM}(#1)} 
\def\cmone#1.{\overline{\operatorname {NM}}(#1)}

\def\coef#1.{\frac{(#1-1)}{#1}}
\def\vit#1.{D_{\langle #1 \rangle}}
\def\mm#1.{\overline {M}_{0,#1}}
\def\H1#1.{H^1(#1,{\ring #1.})}
\def\ac#1.{\overline {\mathbb F}_{#1}}

\def\adj#1.{\frac {#1-1}{#1}}
\def\spn#1.{\overline{#1}}
\def\pek#1.#2.{\Cal P^{#1}(#2)}
\def\plk#1.#2.{\Cal P^{\leq #1}(#2)}
\def\ev#1.{\operatorname{ev_{#1}}}
\def\ilist#1.{{#1}_1,{#1}_2,\dots}
\def\bminv#1.{(\nu_1,s_1;\nu_2,s_2;\dots ;\nu_{#1},s_{#1};\nu_{r+1})}
\def\zinv#1.{(\nu_1,s_1;\nu_2,s_2;\dots ;\nu_{#1},s_{#1};0)}
\def\iinv#1.{(\nu_1,s_1;\nu_2,s_2;\dots ;\nu_{#1},s_{#1};\infty)}

\def\scr #1.{\mathcal #1}

\def\llist#1.#2.{{#1}_1,{#1}_2,\dots,{#1}_{#2}}
\def\ulist#1.#2.{{#1}^1,{#1}^2,\dots,{#1}^{#2}}
\def\lomitlist#1.#2.{{#1}_1,{#1}_2,\dots,\hat {{#1}_i}, \dots, {#1}_{#2}}
\def\lomitlistz#1.#2.{{#1}_0,{#1}_1,\dots,\hat {{#1}_i}, \dots, {#1}_{#2}}
\def\loc#1.#2.{\Cal O_{#1,#2}}
\def\fderiv#1.#2.{\frac {\partial #1}{\partial #2}}
\def\deriv#1.#2.{\frac {d #1}{d #2}}

\def\map#1.#2.{#1 \longrightarrow #2}
\def\rmap#1.#2.{#1 \dasharrow #2}
\def\emb#1.#2.{#1 \hookrightarrow #2}
\def\non#1.#2.{\text {Spec }#1[\epsilon]/(\epsilon)^{#2}}
\def\Hi#1.#2.{\text {Hilb}^{#1}(#2)}
\def\sym#1.#2.{\operatorname {Sym}^{#1}(#2)}
\def\Hb#1.#2.{\text {Hilb}_{#1}(#2)}
\def\Hm#1.#2.{\Hom_{#1}(#2)}
\def\prd#1.#2.{{#1}_1\cdot {#1}_2\cdots {#1}_{#2}}
\def\Bl #1.#2.{\operatorname {Bl}_{#1}#2}
\def\pl #1.#2.{#1^{\otimes #2}}
\def\mgn#1.#2.{\overline {M}_{#1,#2}}
\def\ialist#1.#2.{{#1}_1 #2 {#1}_2, #2\dots}
\def\pair#1.#2.{\langle #1, #2\rangle}
\def\vandermonde#1.#2.{\left|
\begin{matrix}
1 & 1 & 1 & \dots & 1\\
{#1}_1 & {#1}_2 & {#1}_3 & \dots & {#1}_{#2}\\
{#1}_1^2 & {#1}_2^2 & {#1}_3^2 & \dots & {#1}_{#2}^2\\
\vdots & \vdots & \vdots & \ddots & \vdots\\
{#1}_1^{#2-1} & {#1}_2^{#2-1} & {#1}_2^{#2-1} & \dots & {#1}_{#2}^{#2-1}\\
\end{matrix}
\right|
}
\def\vandermondet#1.#2.{\left|
\begin{matrix}
1 & {#1}_1   & {#1}_1^2 & \dots & {#1}_1^{#2-1}\\
1 & {#1}_2   & {#1}_2^2 & \dots & {#1}_2^{#2-1}\\
1 & {#1}_3   & {#1}_3^2 & \dots & {#1}_3^{#2-1}\\
\vdots & \vdots & \vdots & \ddots & \vdots\\
1 & {#1}_{#2}& {#1}_{#2}^2 & \dots & {#1}_{#2}^{#2-1}\\
\end{matrix}
\right|
}
\def\gr#1.#2.{\mathbb{G}(#1,#2)}

\def\alist#1.#2.#3.{{#1}_1 #2 {#1}_2 #2\dots #2 {#1}_{#3}}
\def\zlist#1.#2.#3.{#1_0 #2 #1_1 #2\dots #2 #1_{#3}}
\def\lomitlist30#1.#2.#3.{{#1}_0,{#1}_1 #2 \dots #2\hat {{#1}_i} #2\dots #2 {#1}_{#3}}
\def\lmap#1.#2.#3.{#1 \overset{#2}{\longrightarrow} #3}
\def\mes#1.#2.#3.{#1 \longrightarrow #2 \longrightarrow #3}
\def\ses#1.#2.#3.{0\longrightarrow #1 \longrightarrow #2 \longrightarrow #3 \longrightarrow 0}
\def\les#1.#2.#3.{0\longrightarrow #1 \longrightarrow #2 \longrightarrow #3}
\def\res#1.#2.#3.{#1 \longrightarrow #2 \longrightarrow #3\longrightarrow 0}
\def\Hi#1.#2.#3.{\text {Hilb}^{#1}_{#2}(#3)}
\def\ten#1.#2.#3.{#1\underset {#2}{\otimes} #3}
\def\lomitlist30#1.#2.#3.{{#1}_0 #2 {#1}_1 #2 \dots #2 \hat {{#1}_i} #2 \dots #2 {#1}_{#3}}
\def\mderiv#1.#2.#3.{\frac {d^{#3} #1}{d #2^{#3}}}

\def\Hom{\operatorname{Hom}}

\def\kod{\operatorname{kod}}

\def\deg{\operatorname{deg}}

\def\rest{\operatorname{res}}

\def\e{\Cal E}

\def\e1{E_1}
\def\e2{E_2}

\def\mapdown#1{\big\downarrow\rlap{$\vcenter{\hbox{$\scriptstyle#1$}}$}}

\def\mapse#1{
{\vcenter{\hbox{$\mathop{\smash{\raise1pt\hbox{$\diagdown$}\!\lower7pt
\hbox{$\searrow$}}\vphantom{p}}\limits_{#1}\vphantom{\mapdown{}}$}}}}

\def\VR#1.{height#1pt&\omit&&\omit&&\omit&&\omit&&\omit&\cr}

\def\VRT#1.{height#1pt&\omit&&\omit&\cr}


%% file: symplectic.bbl
\providecommand{\bysame}{\leavevmode\hbox to3em{\hrulefill}\thinspace}
\providecommand{\MR}{\relax\ifhmode\unskip\space\fi MR }
\providecommand{\MRhref}[2]{%
  \href{http://www.ams.org/mathscinet-getitem?mr=#1}{#2}
}
\providecommand{\href}[2]{#2}
\begin{thebibliography}{BHPdV04}

\bibitem[Ati57]{Atiyah57}
M.~Atiyah, \emph{Vector bundles over an elliptic curve}, Proc. Lond. math. Soc.
  \textbf{27} (1957), 414--452.

\bibitem[Bar84]{B84}
R.~Barlow, \emph{A simply connected surface of general type with $p_g=0$},
  Invent. Math. (1984), no.~79, 293--301.

\bibitem[BC09]{BC09}
I.~Bauer and F.~Catanese, \emph{{Burniat surfaces I: fundamental groups and
  moduli of primary Burniat surfaces}}, arXiv:0909.3699, 2009.

\bibitem[BCP10]{BCP10}
I.~Bauer, F.~Catanese, and R.~Pignatelli, \emph{{Surfaces of general type with
  geometric genus zero: a survey}}, arXiv:1004.2583, 2010.

\bibitem[BHPdV04]{BHPV04}
W.~Barth, K.~Hulek, C.A.M. Peters, and A.Van de~Ven, \emph{Compact complex
  surfaces}, second ed., Ergebnisse der Mathematik und ihrer Grenzgebiete.,
  vol.~4, Springer-Verlag, Berlin, 2004.

\bibitem[Bir99]{Biran99}
P.~Biran, \emph{A stability property of symplectic packing}, Invent. Math.
  \textbf{136} (1999), no.~1, 123--155.

\bibitem[BP83]{BP83}
W.~Barth and C.~Peters, \emph{Automorphisms of {E}nriques surfaces}, Invent.
  Math. \textbf{73} (1983), no.~3, 383--411.

\bibitem[Buc99]{Buch99}
N.~Buchdahl, \emph{On compact {K}\"ahler surfaces}, Ann. Inst. Fourier
  (Grenoble) \textbf{49} (1999), no.~1, vii, xi, 287--302.

\bibitem[Cat09]{Cat09}
F.~Catanese, \emph{Canonical symplectic structures and deformations of
  algebraic surfaces}, Commun. Contemp. Math. \textbf{11} (2009), no.~3,
  481--493.

\bibitem[CC93]{CC93}
F.~Catanese and C.~Ciliberto, \emph{Symmetric products of elliptic curves and
  surfaces of general type with {$p\_g=q=1$}}, J. Algebraic Geom. \textbf{2}
  (1993), no.~3, 389--411.

\bibitem[CL97]{CL97}
F.~Catanese and C.~LeBrun, \emph{On the scalar curvature of {E}instein
  manifolds}, Math. Res. Lett. \textbf{4} (1997), no.~6, 843--854.

\bibitem[Gar06]{G06}
L.F. Garc{\'{\i}}a, \emph{Seshadri constants on ruled surfaces: the rational
  and the elliptic cases}, Manuscripta Math. \textbf{119} (2006), no.~4,
  483--505.

\bibitem[Lam99]{L99}
A.~Lamari, \emph{Le c\^one k\"ahl\'erien d'une surface}, J. Math. Pures Appl.
  (9) \textbf{78} (1999), no.~3, 249--263.

\bibitem[Li08]{Li08}
T.-J. Li, \emph{The space of symplectic structures on closed 4-manifolds},
  Third {I}nternational {C}ongress of {C}hinese {M}athematicians. {P}art 1, 2,
  AMS/IP Stud. Adv. Math., 42, pt. 1, vol.~2, Amer. Math. Soc., Providence, RI,
  2008, pp.~259--277.

\bibitem[LL01]{LiLuo01}
T.-J. Li and A.-K. Liu, \emph{Uniqueness of symplectic canonical class, surface
  cone and symplectic cone of 4-manifolds with {$B^+=1$}}, J. Differential
  Geom. \textbf{58} (2001), no.~2, 331--370.

\bibitem[LP07]{LP07}
Y.~Lee and J.~Park, \emph{A simply connected surface of general type with $p_g
  = 0$ and $k^2 = 2$}, Invent. Math. (2007), no.~170, 483Ð505.

\bibitem[PPS09a]{PPS09}
H.~Park, J.~Park, and D.~Shin, \emph{A simply connected surface of general type
  with {$p_g=0$} and {$K^2=3$}}, Geom. Topol. \textbf{13} (2009), no.~2,
  743--767.

\bibitem[PPS09b]{PPS09a}
\bysame, \emph{A simply connected surface of general type with {$p_g=0$} and
  {$K^2=4$}}, Geom. Topol. \textbf{13} (2009), no.~3, 1483--1494.

\bibitem[R{\c{S}}09]{RS09}
R.~R{\u{a}}sdeaconu and I.~{\c{S}}uvaina, \emph{Smooth structures and
  {E}instein metrics on {$\Bbb{CP}^2\#5,\,6,\,7\overline{\Bbb{CP}^2}$}}, Math.
  Proc. Cambridge Philos. Soc. \textbf{147} (2009), no.~2, 409--417.

\bibitem[STY02]{STY02}
I.~Smith, R.~P. Thomas, and S.-T. Yau, \emph{Symplectic conifold transitions},
  J. Differential Geom. \textbf{62} (2002), no.~2, 209--242.

\end{thebibliography}
